\documentclass[preprint,12pt]{elsarticle}
\usepackage{natbib}
\usepackage{amsmath}
\usepackage{amsfonts}
\newtheorem{ex}{Example} 
\newtheorem{them}{Theorem} 
\newtheorem{rk}{Remark} 
\newtheorem{lema}{Lemma} 
\newtheorem{props}{Proposition} 
 
\newtheorem{defn}{Definition} 
\newproof{pf}{Proof}
\newproof{pot1}{Proof of Theorem \ref{thmchatt}}
 
\newcommand{\R}{\mathbb{R}}
\newcommand{\N}{\mathbb{N}}

\newcommand{\K}{\mathbb{K}}
\begin{document}

\begin{center}
{\bf \Large On the Bellman equation in recursive  stochastic dynamic programming 
with the  CES aggregator}
\end{center}

\begin{center}  
{\bf  Anna Ja\'skiewicz$^{a}$ and  Andrzej S. Nowak$^{b}$ } 
\end{center}  
$^{a}$Faculty of Pure and Applied Mathematics, Wroc{\l}aw University of Science and 
Technology, Wroc{\l}aw, Poland, {\it email: anna.jaskiewicz@pwr.edu.pl}\\
\noindent $^{b}$Institute of Mathematics, University of Zielona G\'ora,  Zielona G\'ora, Poland,
{\it email: a.nowak@im.uz.zgora.pl} \\

\noindent {\bf Abstract} 
In this paper we investigate discrete-time infinite horizon 
Markov decision processes (dynamic programming models)  with
recursive utilities  defined by the classical \emph{CES} aggregator 
and  the certainty equivalent  operator being the negative of the entropic risk measure.
Under mild conditions on the primitive data  we prove 
existence and uniqueness of a solution to the Bellman equation. Moreover, we show that
for  any given stationary policy, the Koopmans equation  
has a unique solution. These two facts imply that a measurable selector of maxima in the Bellman equation  
is an optimal stationary policy. 
The methods of our proofs are applied to study solutions 
to the Bellman equation for other certainty equivalents, including risk-neutral case.\\


\section{Introduction} 

The techniques of dynamic programming are widely used in many areas of economics.
This is very well documented in  books 
\cite{acemo,miao,ss25,ss,stach} and \cite{slp}.

The time additive preference relation with discounting,   due to its simplicity, 
has attracted  attention of many researchers, see   
\cite{bb,slp} for further references. 
 An axiomatic foundation  of the additive discounted utility was given  in \cite{koop}.  The problem 
with time additive discounted utility is that the elasticity of intertemporal substitution is mathematically 
the reciprocal of a relative risk aversion coefficient. 
A critical review of the time
additive preferences can be found in \cite{frk} and its limitation are listed in 
Section 7.2.1.2 in \cite{ss25}. These unfavourable comments were 
 the main reason for introducing the concept of non-additive recursive utility, 
see e.g. \cite{bb,ez,ez1,miao,w1}. 
  
In this paper we study a class of non-additive
recursive utilities  in  dynamic decision  models  within    
a Markov environment. Our basic definitions are similar to those 
in \cite{mm,rens} and   \cite{asj,bjjet}.
This is because we are interested in the use of these utilities 
in the context of dynamic programming.

To describe our main objectives we  introduce necessary mathematical objects. More detailed and
rigorous definitions will be given in the sequel. 
Here we assume that all functions 
under consideration are   bounded and Borel measurable.
  
The sets of all real numbers (positive integers) are denoted by $\mathbb{R}$ ($\mathbb{N}$).
We consider a Markov decision process (\emph{MDP}) with Borel state and action spaces denoted by $S$ and $A$, respectively. A  Markov policy of an  agent is a sequence $\pi=(f_n)$  
of measurable mappings $f_n:S\to A.$ The transition probability
from state $s_n$ to state $s_{n+1}$  is denoted by 
$q(\cdot|s_n, f_n(s_n)).$  
Suppose that $u(s_n,f_n(s_n))$
is a  utility of an agent in period $n.$ 
Here $u:S\times A\to \mathbb{R}$ is a bouded Borel measurable function.   
For any  policy $\pi$ and $\tau\in \mathbb{N},$  we write  $\pi^\tau= (f_\tau,f_{\tau+1},...),$ $\pi^1=\pi,$ and  
define the expected utility from period $\tau$ onwards as follows
$$ J(\pi^\tau)(s_\tau): = 
(1-\beta)\mathbb{E}^{\pi^{\tau}}_{s_\tau}\left[\sum_{n=\tau}^\infty \beta^{n-1}u(s_n,f_n(s_n))\right],
$$
where $\beta \in (0,1)$ is a discount factor, $\mathbb{E}^{\pi^\tau}_{s_\tau}$ 
is the expectation operator on the space of state sequences constructed 
by the Ionescu-Tulcea theorem, see \cite{bs} or \cite{hl}.
Assume that we deal with a stationary case, i.e., $f_n = f$ for all $n\in \mathbb{N}.$ 
Then, we put $J(f)=J(\pi)$ and it holds
$$
J(f)(s) =(1-\beta)u(s,f(s)) +\beta \int_SJ(f)(s')q(ds'|s,f(s))\quad\mbox{for all}\quad s\in S.
$$

Let $W(x,y):= (1-\beta)x+\beta y$ where $x,y\in \mathbb{R}$
and  $E(v)(s,f(s)):= \int_Sv(s')q(ds'|s,f(s))$ for a function $v:S\to\R.$
Then   $v = J(f)$ is the unique solution of the equation
\begin{equation} 
\label{ke1}
v(s)=W(u(s,f(s)), E(v)(s,f(s)))\quad\mbox{for all}\quad s\in S.
\end{equation}
This equation is a stationary version of the Koopmans equation 
in the theory of recursive preferences \cite{koop,koopetal}.
The function $W$ is a special case of an aggregator and $E(v)(s,f(s))$ is the expected value of $v$
with respect to the probability measure $q(\cdot|s,f(s)).$ The quantity   $E(v)(s,f(s))$
is a sort of conditional certainty equivalent of $v$ given the state $s$ and the action $f(s)$ of an agent. 
Let $E(v)(s,a):= \int_S v(s')q(ds'|s,a).$ The Bellman equation in this case is
\begin{equation}
\label{be1}
v(s) =\max_{a\in A} W(u(s,a),E(v)(s,a)) = 
\max_{a\in A}\left[(1-\beta)u(s,a)+\beta E(v)(s,a)\right], \quad s\in S.
\end{equation}
It obviously has a unique bounded solution \cite{black}.  

In the literature,  non-additive recursive preferences 
are mainly described on sequences of temporal utilities with the use of a shift operator  \cite{bb}. 
Other authors use a framework that is closer to ours, defining utilities using an aggregator 
and the certainty equivalent, which is a quasi-arithmetic mean, see, e.g.  \cite{mm,miao,rens}.

Let $\mathbb{R}_+=[0,\infty)$ and 
$U: \mathbb{R}_+\to\mathbb{R}_+$ be an increasing concave function. 
A quasi-arithmetic conditional certainty equivalent of a bounded non-negative function $w$ 
for a given stationary policy $f$ induced by $U$ 
is defined as
\begin{equation}
\label{cceu1}
M(w)(s,f(s)):= U^{-1}\left(\int_S U(w(s'))q(ds'|s,f(s))\right),\quad s\in S.
\end{equation}
We also put
\begin{equation}
\label{cceu2}
M(w)(s,a):= U^{-1}\left(\int_S U(w(s'))q(ds'|s,a)\right).
\end{equation}
A function $v_f$ is a solution to the Koopmans equation, if 
$$v_f(s)= W(u(s,f(s)), M(v)(s,f(s)))\ \ \mbox{for all}\ \ s\in S,$$
and  $v$ is a solution to the Bellman equation, if 
$$v(s)=  \max_{a\in A} W(u(s,a), M(v)(s,a))\ \ \mbox{for all}\ \ s\in S.$$
The above equations are direct equivalents of (\ref{ke1}) and (\ref{be1}).

Theoretical results  on  recursive utilities have been studied in many papers,
see for instance \cite{becketal,bsta,chris, ez,hschemca,hs,mm,w1}, 
where the emphasis was put on existence and uniqueness and, see for instance \cite{brp,ez1,ls}, where
some applications are demonstrated. 
Dynamic programming models including recursive utilities have been recently reported in
\cite{balb,bv,bvv} and \cite{rens} (consult also with Chapters 7 and 8 in  \cite{ss}).

In this paper we study existence and uniqueness of a solution 
to the Bellman equation in a model with recursive utility defined with the aid of the \emph{CES} aggregator
and the conditional certainty equivalent induced by an exponential function. Namely, 
\begin{equation}
\label{rsce}
W(x,y):= 
 \left[(1-\beta)x^{p} + \beta
y^{p}\right]^{\frac 1{p}},
\end{equation}
where $1>p\not=0.$
If $U(z)= re^{rz} -r$ for some $r<0$ and for all $z\ge 0,$ then  (\ref{cceu1})  and (\ref{cceu2}) 
become 
$$M(v)(s,f(s))= \frac{1}{r}\ln \int_S e^{rv(s')}q(ds'|s,f(s))\ \  \mbox{and}\ \  
M(v)(s.a)= \frac{1}{r}\ln \int_S e^{rv(s')}q(ds'|s,a).
$$
Clearly, we plug 
$x= u(s,f(s))$ or $x=u(s,a)$ and $y= M(v)(s,f(s))$ or $y= M(v)(s,a)$ into (\ref{rsce}).

The \emph{CES} aggregator allows to model a lifetime utility by combining current utility 
and the certainty equivalent of future utility into a single recursive equation. 
It separates the elasticity of intertemporal substitution $\frac{1}{1-p}$ 
from the absolute risk aversion parameter $-r.$  
The quantity $-M(v)$ is also known as the entropic risk measure \cite{fs}. Therefore, Sargent and Stachurski
\cite{ss25} call $M$ an entropic certainty equivalent operator.  The recursive utility
obtained in this manner is a member of  a wider class of Kreps-Porteus preferences \cite{kp,kp2}.
Namely, Kreps and Porteus \cite{kp,kp2} provided the foundation which enabled 
Epstein and Zin \cite{ez,ez1} and Weil \cite{w1}
to generalise iso-elastic utility. Further comments and historical remarks 
can be found in \cite{backus,bb,hs,miao,ss25}.

The entropic certainty equivalent operator, used recursively in every step,  
was applied  first by Hansen and Sargent \cite{hsar}     
to solve a linear-quadratic Gaussian control problem. 
Their  approach  inspired research   in other areas, e.g.,  in stochastic optimal economic growth theory 
\cite{bjjet}, dividend payout problem in insurance \cite{bjime} and in  Markov controlled models  \cite{asj}.
An example belonging to the class of recursive utilities studied in this paper
was solved by Weil \cite{w2}, who
found  it appealing in a  study of  precautionary savings and the permanent income hypothesis. 

The growing  popularity of recursive preferences gives rise to 
the expanding domain of their
applications in macroeconomics and finance. The reader is referred in this matter, among others, 
to \cite{backus,tal,ss}.
For instance, Backus et al. \cite{backus} studied implications 
of risk and ambiguity on business cycle fluctuations. 
They demonstrated that 
risk sensitive preferences help to identify the sources of aggregate fluctuations and 
dynamics of asset prices. 
Tallarini \cite{tal}, on the other hand, showed that increasing risk aversion significantly improves the model's 
asset market predictions. Moreover, as argued in  \cite{hsrob},
risk sensitive preferences are also attractive, because they can  be used
to model preferences for robustness. Within  such a framework, 
the risk aversion parameter  becomes  the robustness parameter.
A larger degree of risk aversion is identical to a larger degree of concerns for robustness. 
For instance,  in \cite{bism} the  preferences for robustness allow to identify animal spirits type of behaviour 
in the business cycle models.

Despite the substantial progress on recursive utility, 
less is known about its implications for dynamic programming in economics. 
B\"auerle and Ja\'skiewicz \cite{bjjet} consider risk-sensitive 
economic growth models with time additive aggregator. Ren and Stachurski
\cite{rens} study the \emph{CES} aggregators as in this paper, 
but with the certainty equivalent induced by a power function. 
Some dynamic programming problems were treated in \cite{balb,bv,bvv}, but they do not cover
the case in this paper. 
Dynamic programming with recursive utility in economics was commenced  by Streufert \cite{str} 
and Ozaki and Streufert \cite{os}, 
who introduced  a notion of biconvergence. This condition means that utility values can be 
approximated by increasing and
decreasing orbits. 
Similar techniques were used earlier in the context of dynamic programming 
for classical payoff criteria \cite{ber, black1,schal}. 
Here, it is worth mentioning that   Denardo \cite{den} 
was one of the first authors, who examined 
dynamic programming with recursive preferences.

In operations research and control theory, the payoff criteria  
involving the entropic risk measure were initiated 
by Howard and Matheson \cite{hm} and Jacobson \cite{jaco}.  
It should be noted however that their approach is different, not recursive. Namely, they use
the certainty equivalent (induced by an exponential function) defined on state-action processes. 
Then, the  classical additive discounted utility  is  transformed into an exponential multiplicative form. 
With such an approach  optimal policies are derived from a system of Bellman equations 
and are very often non-stationary. 
For the details the reader is referred to \cite{bjmmor,w}, where the results and applications are provided.

Our results involving the recursive preferences are presented for 
an \emph{MDP} model.
It is worth emphasising that \emph{MDPs} 
play an important  role in operations
research, economics and finance,
see \cite{acemo,bs,hs,hl,miao,ss, stach}.  
We assume that the per-period utilities 
are bounded and impose mild continuity-compactness assumptions. 
We show that the dynamic programming operator  is monotone. In this
way we obtain a solution to the Bellman equation. 
To show its uniqueness we prove that the dynamic programming operator 
is strongly subhomogeneous. In contrast to  \cite{mmtarski}, 
we directly utilise the  form of the time aggregator.
Then, the uniqueness is a consequence of this property. 
We work with the most general possible class of functions,
that is, with the upper semianalytic functions. 
This allows properly define a dynamic programming operator that maps  
an upper  semianalytic function into the same class.
We also prove that, for  any given stationary policy, 
the equation of Koopmans \cite{koop} has a unique solution. 
This second uniqueness result is needed to conclude that a measurable selector of maxima 
in the Bellman equation  is an optimal stationary policy. 
A similar dynamic programming  model is considered in \cite{bvv}. It is fairly general, 
but the underlying assumption for uniqueness of a solution to the Bellman equation 
is one of our  main results saying that for each stationary policy the solution to 
the 	Koopmans  equation is unique.

The remainder of this paper is organised as follows. 
Section \ref{preliminary} introduces basic notation and lemmas. 
In Section \ref{model} 
we define recursive utility in our framework. 
The various sets of basic assumptions and corresponding main results are put in 
Section \ref{bellmany}, where we prove the existence and uniqueness of a solution 
to the Bellman and Koopmans equations. Their proofs are postponed to Appendix (Section \ref{appendix}).
Section \ref{sec5}  demonstrates that these techniques can be applied to models with 
other certainty equivalents, whilst Section \ref{comm} comments on our assumptions.  
Finally, Section \ref{sec6} contains concluding remarks.

\section{Preliminaries} \label{preliminary}
Let $S$ be  a Borel space, i.e., a non-empty Borel subset of a complete separable metric space.
We assume that $S$  is equipped with its Borel $\sigma$-algebra $\mathcal{B}(S).$

By a Borel   transition probability from $S$
to a Borel space $Y$ we mean a function
$\xi: \mathcal{B}(Y ) \times S\to  [0, 1]$ such that, for each $B\in\mathcal{B}(Y),$
$\xi(B,\cdot)$ is
a Borel measurable function on $S$ and $\xi(\cdot,s)$ is a probability measure on $\mathcal{B}(Y )$ 
for each $s\in S.$ We shall write $\xi(B|s)$ for $\xi(B,s).$

An analytic set in $Y$
is a projection of a Borel subset of $Y\times Y$ on the horizontal axis. 
There are analytic sets in $Y$ that are not Borel. However, for any analytic set 
$L_1 \subset Y$and any probability measure $\mu$ on $\mathcal{B}(Y)$ there exists a Borel set $L_2 \subset Y$ 
such that $\mu(L_1 \triangle   L_2)=0.$ There are also other definitions of analytic subsets 
of a Borel space $Y,$ see  Chapter 7 in \cite{bs} and \cite{bfo}.
A set $E\subset Y$ is universally measurable, if $E$ is in the completion of the Borel
$\sigma$-algebra with respect to every probability measure on $\mathcal{B}(Y).$
Clearly, every Borel set is analytic and every analytic set is universally measurable. 

It is said that  $g:Y\to \mathbb{R}$ is an {\it upper semianalytic} function,
if for any $c\in \mathbb{R}$ the set
$\{ y\in Y:  g(y)>c \}$ is analytic.
It is said that  $g:Y\to \mathbb{R}$ is a {\it lower semianalytic} function, 
if for any $c\in \mathbb{R}$ the set
$\{ y\in Y:  g(y)<c \}$ is analytic. It is known that
for any upper semianalytic (or lower semianalytic) function $g$ and any probability measure
$\mu$ on $\mathcal{B}(Y)$ there exists a Borel measurable function $g': Y \to \R$ such that
$g=g'$ on a Borel  set $D \subset Y$ with $\mu(D)=1.$
A function $g: Y \to \mathbb{R}$ 
is universally measurable, if the inverse image $g^{-1}(B)$ is universally measurable for any Borel set $B.$ 
Obviously, every upper or lower semianalytic function is universally measurable.
A detailed discussion and bibliographical notes  on semianalytic and universally measurable functions 
are contained in Chapter 7 in \cite{bs}. 

We write   $\mathcal{A}$  [$\mathcal M$] ($\mathcal{U}$) for the  space of all bounded 
upper semianalytic
[Borel measurable] (upper semicontinuous) functions  $ w:S\to \mathbb{R}$
and use  $\mathcal{C}$  to denote  the subspace of   $\mathcal U$  consisting of all continuous functions. 
By   $\|w\| $  we denote the supremum norm for any 	$w\in \mathcal{A}.$ 	
Fix $\epsilon>0$ and $B >\epsilon.$ We write 	
$\widehat{\mathcal{A}}$  to denote the  space of all functions in $\mathcal{A}$ 
with the values in $[\epsilon,B].$ In the same manner, we  define 
$\widehat{\mathcal{M}},$  $\widehat{\mathcal{U}}$ 
and $\widehat{\mathcal{C}}$ as subsets of $\mathcal M,$ $\mathcal U$ and $\mathcal C,$ respectively.  
Clearly, 
$\widehat{\mathcal{C}}\subset \widehat{\mathcal{U}}\subset \widehat{\mathcal{M}}\subset \widehat{\mathcal{A}}.$

Let $Z$ be a separable metric space. By
$\varphi$ we denote a correspondence  from $S$ to $Z.$ For any $D\subset Z$ we
put
$$
\varphi^{-1}(D):=\{s\in S: \varphi(s)\cap D\not= \emptyset\}.
$$
If $\varphi^{-1}(D)$ is a closed (an open) subset in $S$
for each closed (open) subset $D$ of $Z,$ then $\varphi$ is said to be {\it upper
(lower) semicontinuous}. If  $\varphi$  is both upper and lower semicotinuous, then it is called {\it continuous}.

The following facts can be found in Appendix D in \cite{hl}.
For part (a), consult with Corollary 1 in \cite{bp}, whilst
for part (b), consult with
Theorem 2 on p. 116 in \cite{berge} 
or Lemma 17.30 in \cite{ab}.  

\begin{lema}\label{l1}
Let $S$ and $A$ be Borel spaces and $\mathbb{K}$ be a Borel subset of $S\times A$ such that 
$A(s):= \{a \in A:\   (s,a)\in  \mathbb{K}  \}$ is non-empty and compact for each $s\in S.$ 
Let $w:\mathbb{K}\to \mathbb{R}$ be a bounded Borel measurable function.\\
 (a) If
$w(s,\cdot)$ is upper semicontinuous  on $A(s)$
for every $s\in S,$ then the function
$w^*(s):= \max_{a\in A(s)} w(s,a)$ is Borel measurable and there exists 
a Borel measurable mapping $f:S\to A$ such that
$f(s)\in A(s)$ and $w^*(s)= w(s,f(s))$ for all $s\in S.$\\
(b) If
$w$ is upper semicontinuous  on $\K$ and the correspondence $s\to A(s)$ is upper semicontinuous, then the function
$w^*(s):= \max_{a\in A(s)} w(s,a)$ is upper semicontinuous and there exists a Borel measurable 
mapping $f:S\to A$ such that
$f(s)\in A(s)$ and $w^*(s)= w(s,f(s))$ for all $s\in S.$
\end{lema}

The next result  is   a corollary to Proposition 10.1 in \cite{schal}.

 \begin{lema}\label{l2}
Let $(g_n)$ be a  monotone decreasing sequence of upper
semicontinuous functions on a compact metric space $C$.
Then $\lim_{n\to\infty}\max_{a\in C}g_n(a)= 
\max_{a\in C}\lim_{n\to\infty}g_n(a).$
\end{lema}


\section{The model and recursive utility} \label{model}
 
In this paper we study a {\it Markov decision process } 
defined by the following objects.
\begin{itemize}
\item[(i)]  $S$ is a {\it set of states} and is assumed to be a Borel
space;
\item[(ii)]  $A$ is the {\it action space} and
is also assumed to be a Borel space;
\item[(iii)]$\mathbb{K}$ is a non-empty Borel subset of $S\!\times\! A$.
We assume that for each $s\in S$, the non-empty $s$-section $$A(s) =
\{a\in A: (s,a)\in \mathbb{K}\}$$
of $\mathbb{K}$ represents the {\it set of available actions} in state $s$; 
\item[(iv)]  $q$ is a Borel measurable  transition probability from $\mathbb{K}$
to $S,$ representing the {\it law of motion} among states;
\item[(v)] $u:\mathbb{K}\to \mathbb{R}$ is a Borel measurable non-negative {\it per-period
utility  function} such that $b_0\le u(s,a)\le b_1$ for some $b_1>b_0>0$ and all $(s,a)\in\mathbb{K}$;
\item[(vi)] $\beta\in[0,1)$ is a {\it discount factor.}
\end{itemize}

A {\it policy} of an agent  is a   sequence of Borel measurable mappings
$\pi = (\pi_{n})$ from the history space to the action set. More precisely,
each $\pi_{n}(h_{n})\in A(s_n)$, $n\in\mathbb{N},$ where
$h_{n} = (s_{1},a_{1},\ldots,s_{n-1},a_{n-1},s_{n})$ is the history of the process
up to the $n$-th state. For $n=1,$ $h_1=s_1.$ 
The {\it class of all policies} is denoted by $\Pi$.
Let $\Phi$ be the set of all  Borel measurable
mappings $f$ from $S$ into $A$ such that $f(s)
\in A(s)$ for each $s \in S$.  By Theorem 1 in \cite{bp}, 
if $A(s)$ 
is compact for every $s\in S,$ 
$\Phi\not= \emptyset.$ 
A {\it Markov policy} is a   sequence  
$\pi = (f_{n})$ where each $f_n\in \Phi,$ $n\in\N.$ 
The {\it class of all Markov policies} is denoted by $\Pi_M$.
A {\it stationary policy} is a constant sequence
$\pi = (f,f,\ldots),$ with $f \in \Phi,$ and therefore it
can be identified with the Borel mapping  $f \in\Phi.$

The lifetime utility under any Markov policy is first defined 
for  finite step models. Then  the limit is taken 
as the time horizon tends to infinity. 
A similar construction in  risk-sensitive Markov decision processes involving the additive (affine)  
aggregator and the  entropic certainty equivalent operator is demonstrated in  \cite{asj,bjjet}.

A {\it state-action aggregator} $H$ maps feasible state-action pairs 
$(s,a)\in \mathbb{K}$ and bounded Borel measurable 
functions $v$ into real values $H(s, a, v)$ representing lifetime utility, 
contingent on the current action $a$, the current state $s$ and 
the use of $v$ to evaluate future states. The classical additively 
separable case is implemented by setting:
$$H_{classic}(s,a,v)=(1-\beta)u(s,a)+\beta \int_S v(s')q(ds'|s,a), $$
where $ (s,a)\in \mathbb{K}.$
In this paper we consider a {\it state-action aggregator} $H,$  known as the \emph{CES}
aggregator, which has the form 
\begin{equation}\label{H}
H(s,a,v) = 
\left[ 
(1-\beta)u(s,a)^{p} +\beta\left(M(v)(s,a) 
\right)^{p} \right]^{\frac{1}{ p}},
\ \   (s,a)\in \mathbb{K}, 
\end{equation}
 where  $ p \not=0$, $ p <1$, and  
$$
M(v)(s,a) :=
\frac{1}{r}\ln \int_S e^{rv(y)}q(dy|s,a), \quad r<0.
$$
Let us fix a Markov  policy $\pi = (f_{n}).$ 
We put
\begin{equation}\label{T0}
   \mathcal{T}_{f_n}{\bf 0} (s_n):= (1-\beta)^{\frac{1}{p}}u(s_n,f_n(s_n)),
\end{equation} 	
where $ {\bf 0}(s)\equiv 0$ for all $s\in S.$ 
For any  function $v \in \mathcal{M}$ such that $M(v)(s,a)>0$ for each  $(s,a)\in \mathbb{K}$  
define  
\begin{equation}\label{Top}
\mathcal{T}_{f_n}v(s_n) :=  
\left[(1-\beta)u(s_n,f_n(s_n))^p+\beta\left[
M(v)(s_n,f_n(s_n))\right]^p\right]^{\frac{ 1}{ p}}
\end{equation}
for all $s_n\in S,$  $n\in \mathbb{N}.$ 
Furthermore,  we set   
$$\mathcal{T}_{f_1}\mathcal{T}_{f_2}\cdots \mathcal{T}_{f_n}v(s_1) 
:=\mathcal{T}_{f_1}\circ \mathcal{T}_{f_2}\circ\cdots\circ \mathcal{T}_{f_n}v(s_1),
\quad s_1\in S.$$  
The {\it recursive utility in the finite time horizon} for the initial state  
$s_1\in S$ is defined by the following composition of the operators 
\begin{equation}
\label{utn}
 U_{n}(\pi)(s_1):= 
\mathcal{T}_{f_1}\mathcal{T}_{f_2}\cdots \mathcal{T}_{f_n}{\bf 0}(s_1).
\end{equation}
Here, the  process terminates after $(n-1)$  steps  
with the   terminal utility  $(1-\beta)^{\frac{1}{p}}u(s_n,f_n(s_n)).$  
In Lemma \ref{nierutil}(a) from Appendix A we prove that the sequence  $(U_{n}(\pi)(s_1))$ 
is non-decreasing and 
bounded from above if $p\in(0,1).$ In Lemma \ref{nierutil}(b) we prove that the sequence  $(U_{n}(\pi)(s_1))$ 
is non-increasing and bounded from below if $p<0.$ 
Hence,  the   {\it  recursive utility} of an agent in the infinite time horizon  
is defined as follows
$$U(\pi)(s_1):= \lim_{m\to\infty}U_m(\pi)(s_1), \ \  s_1\in S,\  \pi\in \Pi_M.$$ 
Observe that for a stationary policy $f\in \Phi$ we have
\begin{equation}
\label{fpp1}
U(f)(s)= H(s,f(s), U(f)) = \mathcal{T}_fU(f)(s)\ \ \mbox{for all }\ \ s\in S,
\end{equation}
where $H$ is given in (\ref{H}). Equation (\ref{fpp1}) is known as the {\it Koopmans equation.}
Indeed,  we have
$$ U_{m+1}(f)(s)={\mathcal T}_f  U_{m}(f)(s)= \left[ 
(1-\beta)u(s,f(s))^{p} +\beta\left(M(U_{m}(f))(s,f(s)) 
\right)^{p} \right]^{\frac{1}{ p}},\quad s\in S. $$
Letting $m\to\infty$ and making use of the Lebesgue monotone
convergence theorem and  Lemma \ref{nierutil}, we get the conclusion. 

\begin{rk}\label{all}
In the above definitions we restrict ourselves to Markov policies. 
However, the recursive utilities can be defined for 
any policy in the same manner. 
This is done on purpose
in order to avoid complex notation.  
The proof that Markov policies are enough in dynamic programming is standard
and will be omitted in this paper. 
With respect to this issue consult with, for instance,  p. 190 in \cite{bs} or Section 3 in \cite{hl}.
\end{rk}

\begin{rk}\label{stala}
In this paper we work with different classes of functions whose codomain is $[\epsilon, B].$
If $p<0$, then  $b_0\le(1-\beta)^{\frac 1p} u(s,a)$ for  $(s,a)\in\K$, but 
the constant $B$ can be chosen sufficiently large so that 
$(1-\beta)^{\frac 1p} u(s,a)\le (1-\beta)^{\frac 1p}b_1\le B$
for all $(s,a)\in\K.$
If $p\in (0,1)$ then obviously $(1-\beta)^{\frac 1p} u(s,a)\le b_1$ for  $(s,a)\in\K$, 
but we have to take $\epsilon>0$ such that 
$(1-\beta)^{\frac 1p} u(s,a)\ge (1-\beta)^{\frac 1p} b_0 \ge \epsilon$ for all $(s,a)\in\K.$ 
By suitably chosen constants for the given $p$ in the model, 
all our operators will map the functions with values in  $[\epsilon, B]$ 
in the same class.
\end{rk}

\section{The Bellman equation} \label{bellmany}

\subsection{Existence and uniqueness}  
\label{bellequa}

In the sequel we shall study the Bellman equation corresponding  to the recursive utility generated
by the state-action aggregator $H$ in (\ref{H}).

\begin{defn}
A function $v^*: S\to\mathbb{R}$ is a solution to the {\it Bellman equation}, if for all  $s\in S$
\begin{eqnarray}\label{bell}\nonumber
v^*(s)&=&  \sup_{a\in A(s)}H(s,a,v^*) \\&=& \sup_{a\in A(s)}  \left( (1-\beta)u(s,a)^p 
+\beta\left[\frac{1}{r}\ln\int_S e^{r v^*(s')}
 q(ds'|s,a)\right]^{p}
\right)^{\frac{1}{p}}, 
\end{eqnarray}
and provided the integral in (\ref{bell}) is well-defined.
\end{defn}

In the standard dynamic programming models, with time additive utilities 
and Borel primitive data, solutions to the Bellman equations need not be Borel. 
They belong to the class of upper semianalytic functions, for the details see \cite{bs,bfo}.  
Below we prove  existence and uniqueness of an upper semianalytic  solution to the Bellman equation 
for a general model. This result is indispensable
for having uniqueness of an upper semicontinuous solution or a continuous solution 
to the Bellman equation  in the model satisfying additional continuity-compactness  
assumptions (see conditions (U1)-(U3) and (C1)-(C3) formulated  below).

For $v\in\widehat{\mathcal{A}}$ define 
\begin{equation}
\label{tmax}
\mathcal{T}v(s):= \sup_{a\in A(s)}
H(s,a,v),\ \ s \in S.
\end{equation}
By Lemma \ref{A2}  from Appendix B, the function $\mathcal{T}v\in \widehat{\mathcal{A}}.$ 
Moreover, as in Lemma \ref{monotlem} from Appendix A, one can show that
\begin{equation}\label{monoT}
\mathcal{T}v_1  \le  \mathcal{T}v_2, \quad\mbox{ if } \ v_1\le v_2\ \mbox{ and } \ v_i\in \widehat{\mathcal{A}} 
\mbox{ for }\ i= 1, 2.
\end{equation}
Using this fact  and the well-known 
value iteration algorithm  (see \cite{bs,black1,schal}),
we  state our first result.

\begin{them}\label{thm1} 
There exists a solution  $v^* \in\widehat{\mathcal{A}}$ to the Bellman equation in (\ref{bell}). 
\end{them}

\noindent{\bf Proof } 
Take   $v\equiv \epsilon$ and consider consecutive iterations of $\mathcal T$ with itself. Clearly, 
$\epsilon \le \mathcal{T}\epsilon\le B$ and consequently from
(\ref{monoT}), we conclude that the  sequence of iterations of
$\mathcal T$ with itself, denoted by $(\mathcal{T}^{(k)}\epsilon(\cdot))$, is  increasing and 
$\mathcal{T}^{(k)}\epsilon\in\widehat{\mathcal{A}}$ for every $k\in\N.$ Hence,
the limit $v^*(s):=\lim_{k\to\infty}\mathcal{T}^{(k)}\epsilon(s)$ exists for every $s\in S.$
By Lemma 7.30 in \cite{bs}, $v^*\in \widehat{\mathcal{A}}$. 
Therefore, by the monotone convergence theorem we get
$$v^*(s)\ge  \lim_{k\to\infty}\mathcal{T}^{(k)}\epsilon(s)\ge  \lim_{k\to\infty} H(s,a, \mathcal{T}^{(k-1)}\epsilon)
=H(s,a,v^*), \quad (s,a)\in\K$$ 
and consequently, $v^*\ge \mathcal{T}v^*.$
On the other hand,
$$\mathcal{T}(\mathcal{T}^{(k)}\epsilon)(s)\le \mathcal{T}v^*(s), \quad s\in S,$$
which yields $ v^*\le \mathcal{T}v^*.$ This finishes the proof. $\Box$ \\

To prove uniqueness of a solution to the Bellman equation, we apply some arguments used 
by Marinacci and Montrucchio \cite{mmtarski}, who 
studied uniqueness of fixed points in the theorem of Tarski \cite{tar}. This is related to 
{\it  strong subhomogeneity} property of a dynamic programming operator.  Below we provide 
a precise definition of this property.  

Let $\widehat{\mathcal{F}}$ be some set of real-valued functions on $S$ with values in the interval $[\epsilon, B]$ 
and let $\widehat{\mathcal{T}}$ be an operator
that maps the set  $\widehat{\mathcal{F}}$ into itself. 

\begin{defn}
The operator  $\widehat{\mathcal{T}}$ is strongly subhomogeneous 
in the class of functions $\widehat{\mathcal{F}},$
if   there exists a continuous increasing function
$\psi:(0,1)\to (0,1)$ such that for all $c\in (0,1)$, we have $c<\psi (c)<1$ and     \\
$$\widehat{\mathcal{T}}(cv)(s) \ge \psi(c)\widehat{\mathcal{T}}v(s)$$ for all $s\in S$
and for any $v\in  \widehat{ \mathcal{F}}.$    
\end{defn}

In Appendix C we show that the Bellman operator $\mathcal{T}$ is strongly subhomogeneous
in the class  $\widehat{\mathcal{A}}.$  This fact implies 
uniqueness of  $v^*$ obtained in Theorem \ref{thm1}.  For a proof of the next result consult with Appendix C.
	
\begin{them}\label{thm2}
The function $v^*\in \widehat{\mathcal{A}}$ from Theorem \ref{thm1} 
is the unique solution to the Bellman equation.
\end{them}

\begin{rk}
The class of upper semianalytic functions is the largest class, in which we can have a solution 
to the Bellman equation. If we take a bounded lower semianalytic function $v,$ then from Lemmas 
\ref{A1} and \ref{A2} we may deduce that
then $M(v)$ is lower semianalytic. Hence, $H(s,a)\to (s,a,v)$ is lower semianalytic. 
However, the supremum  in (\ref{bell}) need not be a universally measurable function \cite{bfo}.
\end{rk}

We shall deal with the following conditions.
\begin{itemize}
\item[(U1)]  The set  $A(s)$ 
is compact for every $s\in S$ and the correspondence  $s\to A(s)$ is upper semicontinuous.
\item[(U2)]  The function $u$ is upper semicontinuous on $\mathbb{K}.$ 
\item[(U3)] The transition probability $q$ is Feller, i.e., the function 
$$(s,a) \to \int_S v(s') q(ds'|s,a)$$
is  continuous on $\mathbb{K}$ for each $v \in\mathcal{C}.$ 
\end{itemize}

Note that under (U1)-(U3) and by Lemma \ref{l3} from Appendix B, 
the function $(s,a)\to H(s,a,v)$ is upper semicontinuous if $v\in\widehat{\mathcal{U}}$. Hence, 
the supremum in (\ref{tmax}) can be replaced by the maximum.
Moreover, by Lemma \ref{l3},
$\mathcal{T}v \in  \widehat{\mathcal{U}}.$
Making use of the 
value iteration algorithm  (see \cite{bs,black,schal} for similar approaches), we formulate our 
corresponding results to Theorems \ref{thm1}-\ref{thm2} when (U1)-(U3) hold. 
However, we start our iterations from the upper bound $B,$
not from $\epsilon$ as in the proof of Theorem \ref{thm1}.

\begin{them}\label{thm1u} 
Assume (U1)-(U3). Then, the following holds:\\
(a) the Bellman
equation has a solution $v^* \in\widehat{\mathcal{U}};$\\
(b) $v^*$ is the unique solution to the Bellman equation in the class $\widehat{\mathcal{A}}.$
\end{them}

\noindent{\bf Proof } $(a)$
Let  $v\equiv B$ and consider consecutive iterations of $\mathcal T$ with itself. Clearly, 
$\epsilon \le \mathcal{T}B\le B$ and consequently from
(\ref{monoT}), we conclude that  the  sequence of iterations of
$\mathcal T$ with itself, that is $(\mathcal{T}^{(k)}B(\cdot))$, is  decreasing.  
Moreover,   Lemma \ref{l3} implies that 
every function $\mathcal{T}^{(k)}B(\cdot)$ is upper
semicontinuous.  Hence,
$v^*(s):=\lim_{k\to\infty} \mathcal{T}^{(k)}B(s)$ exists for every $s\in S$ 
and is (by Lemma 2.41 in \cite{ab}) upper semicontinuous. 
Making use of Lemma \ref{l2}, we conclude that $v^*=\mathcal{T}v^*.$ 

$(b)$  One one hand, $v^* =\mathcal{T}v^*$ and 
$v^*\in\widehat{\mathcal{U}}.$ On the other hand, by Theorem \ref{thm2}   the operator
$\mathcal T$ has  a unique fixed point, say $v_1^*\in\widehat{\mathcal{A}},$  i.e.,
$v_1^*=\mathcal{T}v_1^*.$ Moreover, from the proof of Theorem \ref{thm2} we must have
$\|v_1^*-v^*\|=0$. Hence, $v_1^*=v^*.$  $\Box$ \\

We may further strengthen assumptions on primitive data to get a continuous solution. 

\begin{itemize}
\item[(C1)]
The set $A(s)$
is compact for every $s\in S$ and  the correspondence $s\to A(s)$ is continuous.
\item[(C2)] The function $u$ is   continuous on $\mathbb{K}.$ 
 \item[(C3)]=(U3) 
\end{itemize}

Then, we arrive at the following equivalent of Theorem \ref{thm1u}.

\begin{them}\label{thm1c} 
Assume (C1)-(C3). Then, we have:\\
(a) the Bellman
equation has a solution $v^* \in\widehat{\mathcal{C}};$\\
(b) $v^*$ is the unique solution to the Bellman equation in the class $\widehat{\mathcal{A}}.$
\end{them}

\noindent{\bf Proof }  By Theorem \ref{thm1u} we know that $v^*\in\widehat{\mathcal{U}}$ 
and this solution is unique.   
By Lemma \ref{l3c} from Appendix B, for every $k\in\N,$  $\mathcal{T}^{(k)}\epsilon\in \widehat{\mathcal{C}}$ 
and $(\mathcal{T}^{(k)}\epsilon(\cdot))$ is increasing and bounded from above by $B$. Hence,
the limit $v^*_1(s):=\lim_{k\to\infty}\mathcal{T}^{(k)}\epsilon(s)$ exists and is lower semicontinuous.  
From the proof of Theorem \ref{thm2}, we know that $\|\mathcal{T}^{(k)}\epsilon-v^*\|\to 0$ as $k\to\infty.$ 
Thus, $v^*_1=v^*\in \widehat{\mathcal{C}}.$ $\Box$ \\

\subsection{Uniqueness of the Koopmans equation and existence of an optimal policy}  
\label{bellequao}

We first show that the function $U(f)$ for $f\in\Phi$ is the unique solution to the Koopmans equation.   
The proof is  based on the fact that the operator $\mathcal{T}_f$ is strongly subhomogeneous in the class 
$\widehat{\mathcal{M}}.$ This result is applied in the proof of Theorem \ref{thm3} on the existence of 
an optimal stationary policy. The proofs are postponed to Appendix D.

The next result for $p\in (0,1)$ is related to Theorem 1 in \cite{mm}. Marrinacci and Montrucchio \cite{mm} 
showed that for a Thompson-type aggregator and subhomogeneous certainty 
equivalent\footnote{The entropic risk measure operator is subhomogeneous.} the recursive utility is unique.
However, their proof is not stated for negative values of $p,$ when the \emph{CES} aggregator is not of Thompson-type
(see Section \ref{sec6} for the formal definitions).

\begin{them}\label{thm4}
For any stationary policy $f\in \Phi$ the function $U(f)$ is
the unique solution to Koopmans equation (\ref{fpp1}).
\end{them}

\begin{defn}\label{defut2}
A policy $\pi^* $ is optimal, if
$$U(\pi^*)(s)  \ge U(\pi)(s)\quad\mbox{for all}\quad \pi\in \Pi,\ s\in S.$$ 
\end{defn}

\begin{them}\label{thm3}
Assume (U1)-(U3) or (C1)-(C3). Then, there 
exists an optimal stationary policy $f^*\in \Phi$  and  
$$U(f^*)(s)= \mathcal{T}U(f^*)(s)= \mathcal{T}_{f^*}U(f^*)(s)= H(s,f^*(s),U(f^*)) \ \ \mbox{for all}\ \ s\in S.$$
If (U1)-(U3) ((C1)-(C3)) hold, then $U(f^*)\in\widehat{\mathcal{U}}$ ($U(f^*)\in\widehat{\mathcal{C}}$). 
\end{them}

Finally, we would like to mention that 
alternatively, one can use another set of assumptions. 
These conditions imply the existence of a Borel measurable solution to the
Bellman equation. 

\begin{itemize}
\item[(S1)]
For each $s \in  S$, the set $A(s)$
is compact.
\item[(S2)] The function $u(s,\cdot)$ is   continuous on $A(s)$ for each $s\in S.$ 
\item[(S3)] The function $a\to\int_S v(s')q(ds'|s,a)$ is continuous for every $s\in S$ and 
any  function $v\in\mathcal{M}.$
\end{itemize}

We define the operator $\mathcal T$ as in (\ref{tmax}) and under (S1)-(S3) 
we have that $\mathcal{T}:\widehat{\mathcal{M}}\to
\widehat{\mathcal{M}}.$ Indeed, it suffices to note that the function $s\to H(s,a,v)$ 
is Borel measurable for every $a\in A(s)$ and $v\in \widehat{\mathcal{M}}.$ Then, by Lemma \ref{l1}(a)
we obtain $\mathcal{T}v\in \widehat{\mathcal{M}}.$  We have the following fact.

\begin{them}\label{thmset}
Assume (S1)-(S3). Then, the Bellman equation has a unique solution $v^*\in\widehat{\mathcal{M}}.$ Moreover, there 
exists an optimal stationary policy $f^*\in \Phi$  and  
$$U(f^*)(s)= \mathcal{T}U(f^*)(s)= {\cal T}_{f^*}U(f^*)(s)= v^*(s) \ \ \mbox{for all}\ \ s\in S.$$
\end{them}

\section{Applications to other certainty equivalents}\label{sec5}

{\it The Epstein-Zin preferences. } Our approach can  be applied 
to the Epstein-Zin-Weil certainty equivalent \cite{miao,ss25}.
In this case, the operator $M$ is replaced by
$$P(v)(s,a):=\left(\int_S v(s')^d q(ds'|s,a)\right)^{\frac 1d},$$
where $d\not =0.$ Clearly, $d=1$ corresponds to the risk-neutral case. 
If $d>1,$ then the agent is risk-loving and if 
$d<1,$ $d\not=0,$ he is risk-averse.
Therefore, we may assume that $d<1$ and $d\not=0.$  
The constant $1-d$ is called a relative risk aversion parameter \cite{miao}. 
Define the operator $\widetilde T$  either for $v\in\widehat{\mathcal{U}}$ or $v\in\widehat{\mathcal{C}}$
as follows
$$\widetilde Tv(s):=\max_{a\in A(s)} \left((1-\beta)u(s,a)^p+\beta [P(v)(s,a)]^p\right)^{\frac{1}{p}}, \quad s\in S.$$
For the model described in Section \ref{model} via points (i)-(vi), we obtain the following result.

\begin{them} \label{thmp}
Under Assumptions (U1)-(U3) there exists a unique solution  
$v^*\in\widehat{\mathcal{U}}$ to the Bellman equation
\begin{equation}
\label{belpot}
v^*(s)  = \widetilde Tv^*(s)=
\max_{a\in A(s)}\left((1-\beta)u(s,a)^p+\beta [P(v^*)(s,a)]^{p}\right)^{\frac{1}{p}}, \quad s\in S.
\end{equation}
If (C1)-(C3) hold, then $v^*\in\widehat{\mathcal{C}}.$
\end{them}
 
The proof is based on Lemma \ref{SH} from Appendix C for the operator $\widetilde T.$ 
It is not difficult to notice that for 
any constant $c>0$ one has $P(cv)=cP(v).$  
Hence, the analogous formula to (\ref{ji}) holds with $P$ instead of $M$ and the form of function $\psi(\cdot)$
remains the same as in Lemma \ref{SH}.

A few words are in order. 
The model with the Epstein-Zin-Weil preferences was considered by Ren and Stachurski \cite{rens} 
with $d=1-\gamma$ and 
$p=1-\rho.$ Ren and Stachurski \cite{rens} examined three cases: 
\begin{itemize}
\item Case 1: $0<\rho<\gamma<1,$ i.e., $1>p>d>0;$
\item Case 2: $0<\rho<1<\gamma,$ i.e., $1>p>0>d;$
\item Case 3: $1<\rho<\gamma,$ i.e., $0>p>d.$ 
\end{itemize}
They proved that the optimality equation has a unique solution and that there exists  an optimal stationary policy.
However, Ren and Stachurski \cite{rens} did not define the recursive utility for any policy and the optimality 
was defined among the class of stationary policies. The proof is based on the Du's theorem, see \cite{du}.
Cases 1 and 3 and two additional cases:
\begin{itemize}
\item Case 4: $0<\gamma<\rho<1,$ i.e., $1>d>p>0;$
\item Case 5: $1<\gamma<\rho,$ i.e., $0>d>p;$
\end{itemize}
were examined by Ja\'skiewicz and Nowak \cite{jn}. They showed that after some transformation  the problem 
of existence and uniqueness of a solution to the Bellman equation 
boils down to the use of the Banach contraction mapping principle.
Our method, on the other hand,  
works for any values $p$ and $d$ provided that they are different from $0$ 
and the per-period utility function
is bounded. \\

{\it A risk-neutral case. } The proof of Lemma \ref{SH} from Appendix C 
can be applied to the risk-neutral case, that is, when the entropic certainty equivalent operator
$M$ is replaced by the conditional expectation, i.e., $E(v)(s,a)=\int_Sv(s')q(ds'|s,a).$ 
The function $\psi(\cdot)$ remains the same as before.
Define the operator $T$  either for $v\in\widehat{\mathcal{U}}$ or $v\in\widehat{\mathcal{C}}$
as follows
$$Tv(s):=\max_{a\in A(s)} \left((1-\beta)u(s,a)^p+\beta [E(v)(s,a)]^p\right)^{\frac{1}{p}}, \quad s\in S.$$
A simple adaptation of the arguments used in the proofs 
of Theorems \ref{thm1u} and \ref{thm1c} yield the following result for the model given in Section \ref{model}
via points (i)-(vi).

\begin{them} \label{thm0}
Under Assumptions (U1)-(U3) there exists a unique solution  
$v^*\in\widehat{\mathcal{U}}$ to the Bellman equation
\begin{equation}
\label{belneu}
v^*(s)  = Tv^*(s)=
\max_{a\in A(s)}\left((1-\beta)u(s,a)^p+\beta [E(v^*)(s,a)]^p\right)^{\frac{1}{p}}, \quad s\in S.
\end{equation}
If (C1)-(C3) hold, then $v^*\in\widehat{\mathcal{C}}.$
\end{them}

If $s'=g(s,a)$ is a continuous function and $q(\cdot|s,a)=  \delta_{g(s,a)}(\cdot),$ 
then we deal with the {\it deterministic model}.  
Here, $\delta_{s'}$ is the Dirac measure concentrated at the point $s'\in S$.
The assertions of Theorem \ref{thm0} hold and the Bellman equation has the form
\begin{equation}
\label{ourspecialbe} v^*(s)= 
\max_{a\in A(s)}\left((1-\beta))u(s,a)^p+\beta v^*(g(s,a))^p\right)^\frac{1}{p}.
\end{equation}
The above results are stated for any $p<1$ and $p\not=0.$
We would like to point out that in the case $0<p<1$ we can make an important simplification and assume that
$$0\le u(s,a)\le B\ \ \mbox{for all}\ \ (s,a)\in
\mathbb{K}.$$
Here, $u$ need not be bounded from below by a positive constant.
Substituting $w=v^{*p}$ we can write (\ref{belneu}) in the equivalent form
$$ w(s) := 
T_1w(s):= \max_{a\in A(s)}
\left((1-\beta)u(s,a)^p+\beta [E(w^{1/p})(s,a)]^p\right).
$$
Moreover, (\ref{ourspecialbe}) can be equivalently written as 
$$
w(s)= T_2w(s):= \max_{a\in A(s)}
\left((1-\beta))u(s,a)^p+\beta w(g(s,a))\right).
$$
Note that the operator $T$ is not a contraction, whereas
$T_i$  ($i=1,2$) are $\beta$-contractions.

\begin{rk}   (a) From the form of the function $\psi(\cdot)$ derived in  Lemma \ref{SH}  and applied to 
the proofs of Theorems \ref{thm2}, \ref{thm1u}, \ref{thm1c}, \ref{thm4}, \ref{thmp} and \ref{thm0},
we conclude that for any constant $\eta\in [\epsilon,B],$ we have
$$\|\widehat{\mathcal{T}}^{(n)}\eta -v^*\| \le B\left(\frac{1}{t_n} - 1\right),\ \ n\in \mathbb{N},$$
where $t_n =\psi(t_{n-1}),$ and $t_0=\epsilon/B$ and  
$\widehat{\mathcal{T}}=\mathcal{T}$ or $\widehat{\mathcal{T}}=T$ or $\widehat{\mathcal{T}}=\widetilde T.$
If $p\in (0,1)$, then $\psi(\cdot)$ is given in (\ref{numer1})
and if $p<0,$ then $\psi(\cdot)$ is given in (\ref{numer2}).  
The above comments concern both risk-sensitive and risk-neutral cases.

(b) If $p\in (0,1),$ then we can apply the Banach contraction mapping principle to the  operators $T_i$ 
to get a  unique fixed point $w_i^*,$ $i=1,2$. Since
now  $0\le u\le B$, we obtain that
$$\|T_i^{(n)}{\bf 0}- w_i^*\| \le \frac{\beta^n}{1-\beta}
\|T_i{\bf 0}- {\bf 0}\| \le 
\beta^nB^p,\quad \ n \in \mathbb{N}.$$  
Thus, we have a geometric convergence.\footnote{See Section 3.12 in \cite{ab}.}  Clearly, 
$v^*_1=w_1^{*1/p}$ and  $v^*_2=w_2^{*1/p}$ are unique solutions to the Bellman equations 
(\ref{belneu})  and (\ref{ourspecialbe}), respectively.

(c) By Lemma \ref{l1}(a), there exists a selector $f^*\in\Phi$ of the maxima on the right-hand side of (\ref{belpot}).
Proceeding along the lines as in Theorem \ref{thm3} we conclude that $f^*$ is an optimal stationary policy. 
The same remark applies to the risk-neutral case and equation (\ref{belneu}).
\end{rk}

\section{Comments on a lower bound of $u$} \label{comm}

In the model, we assume that $0<\epsilon \le u \le B$.
This condition is  needed when $p<0.$ If $p\in (0,1)$,
it looks too strong. It excludes for example the well-known \emph{DARA} utility function
$u(s,a)= a/(a+1)$, where $s\ge 0$, $a\in [0,s].$  In addition, $u(s,a)= 1 - e^{-a}$ 
is also excluded by this restriction.

In the proof of uniqueness of solution to the Bellman equation, 
we need the property that there exists $\epsilon >0$ such that
for a fixed point $v$ of the  Bellman operator $\mathcal T$
we have $v\ge \epsilon.$ 
We propose the following assumption. 
\begin{itemize}
\item[(A)] Let $u\ge 0$ and $p\in(0,1).$ 
There exists $\epsilon >0$ such that $\epsilon \le \mathcal{T}^{(2)}{\bf 0}(s)$ for all $s\in S.$
\end{itemize}
Here, we use the same letter $\epsilon>0$ as in the previous sections, but obviously it need not be same. 
Note that 
$$\widehat{u}(s):=\mathcal{T}{\bf 0}(s)= \max_{a\in A(s)} (1-\beta)^{\frac{1}{p}}u(s,a)$$
 and 
  $$\mathcal{T}^{(2)}{\bf 0}(s)= \max_{a\in A(s)}\left( 
(1-\beta)u(s,a)^p +\beta [M (\widehat{u})(s,a)]^p    \right)^{\frac{1}{p}}
$$
 
\begin{ex} 
Let  $S=\mathbb{R}_+,$ $A(s)=[0,s]$ for $s\in S.$ 
Assume that $u(s,a)=a/(a+1).$ Then $\widehat{u}(s)= s/(s+1).$ 
The next state is given by the equation
$$s_{t+1}= (s_t - a_t)\rho + \xi_{t+1},$$
where  $\rho>0$ and $(\xi_\tau)$ is an i.i.d. sequence
having a probability distribution $\lambda$ on $S.$ We put
$$\mu=  \int_Se^{r\widehat{u} (s)}
\lambda(ds).$$
Since $\widehat{u}(s)>0$ for $s>0$, we have $0<\mu<1.$
Moreover, since $\widehat{u}$ is increasing and $r<0$,
 $$e^{r\widehat{u}((s-a)\rho +\xi)}
\le e^{r\widehat{u}( \xi)}.$$  
 Hence,
$$\int_S e^{r\widehat{u}(s')}q(ds'|s,a) \le 
\int_Se^{r\widehat{u}(s)} \lambda(ds) =\mu$$
and 
$$M(\widehat{u})(s,a)\ge \frac{\ln\mu}{r}.$$
Thus
\begin{equation}
\label{in2}
\mathcal{T}^{(2)}{\bf 0} (s) \ge \epsilon := \frac{\beta^{\frac{1}{p}}
\ln\mu}{r}>0.
\end{equation}
A similar example is given  in \cite{bvv}, but the probability distribution $\lambda$ is on
a closed interval included in $(0,\infty)$. Hence, the support of $\lambda$  is separated from 0. 
\end{ex} 

\begin{props}\label{p1}
If (A) holds, then the Bellman equation has a unique solution $v^*\in \widehat{\mathcal{A}}.$ 
\end{props}

\noindent {\bf Proof } 
Taking the iterations of $\mathcal{T}$ with itself on the zero function ${\bf 0}$ we conclude as in the proof of Theorem 
\ref{thm1} that the sequence $(\mathcal{T}^{(k)}{\bf 0}(\cdot))$ is increasing and bounded by $B.$ Hence,
the limit $v^*(s):=\lim_{k\to\infty} \mathcal{T}^{(k)}{\bf 0}(s)$ exists for every $s\in S,$  
$v^*\in \widehat{\mathcal{A}}$ is a solution to the Bellman equation. 
In the same manner as in the proof of Theorem \ref{thm2}, we show that $v^*$ is the unique fixed point of $\mathcal T$
in the class of $\widehat{\mathcal{A}}.$  
Let now ${\cal T}v =v$ for some $v\in \mathcal{A}$ and $0\le v\le B.$  
We have  that $\mathcal{T}{\bf 0}\le \mathcal{T}v =v$ and 
$\mathcal{T}^{(2)}{\bf 0}\le \mathcal{T}v =v.$ Since
$\epsilon\le \mathcal{T}^{(2)}{\bf 0}$, we get 
$\epsilon  \le v.$  Hence, $v^*$ is unique. $\Box$

\begin{props} \label{p2}  Assume (A) and (U1)-(U3). Let $f^*\in\Phi$  realise the maximum 
in (\ref{bell}).
If there exists $\epsilon_{f^*}>0$ such that 
$\mathcal{T}_{f^*}^{(2)}{\bf 0}(s)\ge \epsilon_{f^*}$, then $f^*$ is an optimal stationary policy.
\end{props}

\noindent {\bf Proof } From Theorem \ref{thm1u} we know that $v^*\in\widehat{\mathcal{U}}.$ 
Let  $f^*\in\Phi$  such that $\mathcal{T}_{f^*}v^*=\mathcal{T}v^*=v^*.$
Then,  $U(f^*)=\mathcal{T}_{f^*}U(f^*)$ and $U(f^*)\ge \epsilon_{f^*}.$ 
Mimicking the proof of Lemma \ref{SH} we infer that the operator
$\mathcal{T}_{f^*}$ is strongly subhomogeneous in the class of measurable 
functions with values in $[\epsilon_{f^*},B].$ 
Then, $U(f^*)$ is unique in the class of Borel measurable functions with values in $[\epsilon_{f^*},B].$
From (A) we infer that $U(f^*)$ is unique in the class of Borel measurable functions with values in $[0,B].$
The remaining part goes as in the proof of Theorem \ref{thm3}. $\Box$  \\

Above we take  the second iteration of $\mathcal T$ 
and $\mathcal{T}_{f^*}$ on the zero function. However, all conclusions remain valid
if take the $\ell$-th iteration of $\mathcal T$ 
and $\mathcal{T}_{f^*}$ on the zero function with $\ell\in\N.$

\section{ Concluding remarks} \label{sec6}

According to the terminology introduced in \cite{mm}, 
a function $W:\mathbb{R}_+\times \mathbb{R}_+ \to 
\mathbb{R}_+$ is  a Blackwell aggregator if 
$$|W(x,y_1)-W(x,y_2)|\le \beta |W(x,y_1)-W(x,y_2)|\ \  \
\mbox{for all}\ \   x, \ 
y_1, \ y_2 \in \mathbb{R}_+,$$
where $0<\beta <1,$ and $W$ is a Thompson aggregator
if $W$ is concave at 0, that is,
$$W(x,\alpha y) \ge \alpha W(x,y) +(1-\alpha)W(x,0)
\ \  \
\mbox{for all}\ \   x, 
 \ y \in \mathbb{R}_+,\ \ 0<\alpha <1,$$
and $W(x,0)>0 $ for  each $x>0.$
This terminology is motivated by the works of Blackwell \cite{black}  and  Thompson \cite{tho}. 

The \emph{CES} aggregator  is of Blackwell-type, if $p>1.$   
This case is relatively simple and boils down to the Banach contraction mapping principle. 
Therefore, we do not consider it.  If $p\in (0,1),$
then $W$ is of Thompson-type and the Banach contraction mapping principle 
cannot be applied  to study solutions to the Bellman equation. 

The fact that elasticity of intertemporal substitution (\emph{EIS}) 
is $\frac1 {1-p}$ under the \emph{CES} aggregator is significant.
The quantity \emph{EIS} can be derived from data using regression and other methods. 
These estimates vary and suggest the values for $p$ between $-1$ and $-0.25.$ In this matter, 
consult with comments and further remarks on p. 210 in \cite{ss25}.
Here, we are particularly interested in the case with $p<0.$ 
There are relatively a small number of 
results in the literature that deal with $p<0$  within  dynamic programming framework. For example, 
Ren and Stachurski \cite{rens} consider the \emph{CES} aggregator  with  $p<0$ but using a different
certainty equivalent, induced by the power function.  
It should be noted however 
that for negative values of $p$ the aggregator $W$ is neither Blackwell-type nor Thompson-type.

In order to prove the existence of a solution to the Bellman equation we apply the value iteration algorithm, 
the well-known technique  used in  dynamic programming \cite{bs,schal}. It produces a monotone sequence of
value functions. Then, the pointwise limit constitutes the desired solution. 
To prove uniqueness of solution to the Bellman  and to the Koopmans equation for any stationary policy, 
we  introduce the property of strongly subhomogeneous operators used in \cite{krasno} and \cite{mmtarski} 
for the theory of monotone operators. 

The existence and uniqueness of solution to the Bellman and Koopmans equations for general  Thompson aggregators
$W$ were examined in \cite{mm,mmtarski}. For example, Marinacci and Montrucchio in Proposition 22 in \cite{mmtarski} consider the Bellman equation of the form
\begin{equation}
\label{specialbe}
 v(s)= \max_{a\in A(s)}W(u(s,a),v(a)),
\end{equation} 
 where $u\ge \epsilon>0.$  Therorem 1 in \cite{mm}
is a related result to ours on recursive utilities and is devoted to a stochastic model. 
Equation (\ref{specialbe}) is a special case of our equation (\ref{bell}) and concerns
a deterministic decision model with a specific  transition function
$g(s,a)= a \in A \subset S.$
The results in \cite{balb} and \cite{bv,bvv} are based on different assumptions and have no implications
to our model. For example,   Bloise et al. \cite{bvv} deal with very general monotone and continuous aggregators 
and also embrace a stochastic case. But to get uniqueness
of  a solution to the Bellman equation, 
they impose Assumption 5 that the solution to the Koopmans equation is unique 
for any stationary policy, see Proposition 1 in \cite{bvv}. 
In our paper, this property is a consequence of mild
conditions on the primitive data. 

We also emphasise that Du's theorem  \cite{du} on a fixed point 
for concave/convex operators exploited in \cite{rens},
has no practical use  in our model. The function $\phi(x)=(1+ x^p)^{1/p}$ is strictly concave 
on $\mathbb{R}_+$ for $p <1,$  $p\not= 0,$ 
and $M(v)(s,a)$ is concave in $v$, see \cite{fs}.
Since the supremum operation is subadditive, we cannot claim that $\mathcal{T}v$ is  concave or convex  
with respect to $v$ and  apply the fixed point theorem of Du \cite{du}.

\section{Appendix} \label{appendix}

Appendix is divided into four parts. Appendix A contains the proof of monotone properties  
of recursive utilities defined in the finite time horizon. 
Appendix B collects essential facts on dynamic programming operators, whereas in Appendix C we show that
the dynamic programming operator is strongly subhomogeneous.
We also prove that this property is fundamental and implies uniqueness of a solution to the Bellman equation. 
Appendix D is devoted to uniqueness of a solution to the Koopmans equation, which in turn implies  existence 
of a stationary optimal policy. 

 \subsection{Appendix A}
\begin{lema}\label{monotlem} Let $f\in F.$
If $M(v_i)(s,f(s))>0$ for  $i= 1, 2$ 
and  $s\in S,$ 
and $v_1\le v_2,$ then ${\cal T}_{f}v_1 \le  {\cal T}_{f}v_2.$ 
\end{lema}

\noindent{\bf Proof }
Note that, if  $v_1\le v_2,$ then 
$$M(v_1)(s,f(s)) \le M(v_2)(s,f(s)) .$$ 
For $p<0$  we have  
\begin{eqnarray*}	\lefteqn{(1-\beta)u(s,f(s))^p+\beta [M(v_1)(s,f(s))]^p \ge }\\&&
(1-\beta)u(s,f(s))^p+ \beta[M(v_2)(s,f(s))]^p,\quad s\in S. 
\end{eqnarray*}
Since $1/p <0$, we conclude that 
$
{\cal T}_{f}v_1(s)  \le {\cal T}_{f}v_2(s)
$
for all $s\in S.$ 
For $p>0$ the assertion is obvious.
$\Box$


\begin{lema}\label{nierutil} Let $\pi=(f_n)\in\Pi_M.$  For all $m\in\mathbb{N}$ and $s_1\in S,$ it holds\\
$(a)$ $U_{m}(\pi)(s_1) \ge U_{m+1}(\pi)(s_1)\ge \epsilon$, when $p<0$ and \\
$(b)$ $U_{m}(\pi)(s_1) \le U_{m+1}(\pi)(s_1)\le B$, when $p>0.$ 
\end{lema}

\noindent{\bf Proof }  $(a)$ We use an induction argument. Let $\pi' =(f_2,f_3,...).$  Since $p<0$, then for $m=1$
$$U_1(\pi)(s_1)=(1-\beta)^{\frac 1p}u(s_1,f_1(s_1))\ge (1-\beta)^{\frac 1p} \epsilon>\epsilon$$
and since 
$(1-\beta)u(s_1,f_1(s_1))^p+\beta(M(U_1(\pi'))(s_1))^p \ge (1-\beta)u(s_1,f_1(s_1))^p,$ we have
$$
U_2(\pi)(s_1)=[(1-\beta)u(s_1,f_1(s_1))^p+\beta(M(U_1(\pi')))(s_1))^p  ]^{\frac 1p}   \le
 (1-\beta)^{\frac 1p}u(s_1,f_1(s_1))=U_1(\pi)(s_1).$$
Now we show the induction step.
Suppose that 
$U_{m}(\pi')(s_2) \ge U_{m+1}(\pi')(s_2)\ge \epsilon$ for all $s_2\in S.$ Then
$$M(U_m(\pi'))(s_1,f_1(s_1)) \ge M(U_{m+1}(\pi'))(s_1,f_1(s_1)) \ge \epsilon.  $$ 
By Lemma \ref{monotlem} we get
$$
U_{m+1}(\pi)(s_1)= {\cal T}_{f_1}U_{m}(\pi')(s_1)
\ge {\cal T}_{f_1}U_{m+1}(\pi')(s_1)= U_{m+2}(\pi)(s_1)
$$
for all $s_1\in S.$  Moreover, since both  functions $y\to y^p$ and $y\to y^{1/p}$ are decreasing, we get
$$U_{m+2}(\pi)(s_1)={\cal T}_{f_1}U_{m+1}(\pi')(s_1)\ge ((1-\beta)\epsilon^p+\beta\epsilon^p)^{\frac 1p}=\epsilon.$$
The proof of $(b)$ for $p>0$ is simpler. Clearly,
$$U_1(\pi)(s_1)=(1-\beta)^{\frac 1p}u(s_1,f_1(s_1))\le (1-\beta)^{\frac 1p} B<B$$
and $$U_2(\pi)(s_1)=[(1-\beta)u(s_1,f_1(s_1))^p+\beta(M(U_1(\pi'))(s_1))^p]^{\frac 1p} \le 
(1-\beta)^{\frac 1p}u(s_1,f_1(s_1))=U_1(\pi)(s_1).$$
The remaining part is omitted. It is a consequence of an induction argument
and Lemma \ref{monotlem}. $\Box$

 \subsection{Appendix B}
 
We start with some useful facts on upper and lower semianalytic functions.
For a proof of Lemma \ref{A1} consult with Proposition 7.48 in \cite{bs}.
 
\begin{lema} \label{A1} Assume that $g: S\to\R$ is bounded and upper (lower) semianalytic. 
Then $(s,a)\to \int_S g(s')q(ds'|s,a)$ is upper (lower) semianalytic on $\mathbb{K}.$ \end{lema}

\begin{lema} \label{A2} Assume that $v\in \widehat{\mathcal{A}}.$
Then $(s,a)\to H(s,a,v) $
is upper semianalytic and $H(s,a,v)\in [\epsilon, B]$
for every $(s,a)\in\K.$  Moreover,  $\mathcal{T}v\in\widehat{\mathcal{A}}.$
\end{lema}

\noindent{\bf Proof }  First we note that the sum of a Borel measurable function and an upper 
(a lower) semianalytic function is upper (lower) semianalytic.  
We  now show that $M(v)(s,a) =\frac{1}{r}\ln \int_S e^{rv(s')}q(ds'|s,a)$ is upper semianalytic.  
Since $r<0,$ the function $e^{rv(s')}$ is lower semianalytic.  
By Lemma \ref{A1}, the function $\int_S e^{rv(s')}q(ds'|s,a)$
is lower semianalytic on $\mathbb{K}.$ Since $y\to \frac{1}{r}\ln y$ is decreasing, it follows that 
$M(v)(s,a)$  is upper semianalytic on $\mathbb{K}.$

If $p\in (0,1),$ then the functions $y\to y^p$ and
$y\to y^{\frac{1}{p}}$ are increasing. Therefore,
$$(s,a)\to ((1-\beta)u(s,a)^p +\beta [M(v)(s,a)]^p)$$ is upper semianalytic,  and consequently,
$$H(s,a,v)=\left((1-\beta) u(s,a)^p+\beta [M(v)(s,a)]^p\right)^{\frac{1}{p}}$$
is upper semianalytic on $\mathbb{K}.$ 
The second part is obvious for $p\in (0,1)$, since $M(v)(s,a)\in [\epsilon, B]$
for all $(s,a)\in\K.$

If $p<0,$ then the function  
$y\to y^p$ is decreasing. Hence 
$$(s,a)\to ((1-\beta)u(s,a)^p +\beta [M(v)(s,a)]^p)$$
is lower semianalytic. But  $y\to y^{\frac{1}{p}}$ is also 
decreasing. Therefore,
$$H(s,a,v)=\left((1-\beta) u(s,a)^p+\beta [M(v)(s,a)]^p\right)^{\frac{1}{p}}$$
is upper semianalytic.  
For the second part note that $[M(v)(s,a)]^p\in [B^p,\epsilon^p]$ and $u(s,a)^p\in [B^p,\epsilon^p]$
for all $(s,a)\in\K.$ 

By Lemma 7.47 in \cite{bs}, it follows that  $\mathcal{T}v$ is upper semianalytic. $\Box$ \\

If  continuity-compactness assumptions on the primitive data are satisfied,
then our results can be strengthened.    

\begin{lema}\label{l3} Assume that (U1)-(U3) hold.
If $v\in\widehat{\mathcal{U}},$ then the function 
$(s,a)\to H(s,a,v)$ is upper semicontinuous.  Moreover, $\mathcal{T}v\in \widehat{\mathcal{U}}.$
\end{lema}

\noindent{\bf Proof }
Since $r<0,$ the function $s' \to e^{rv(s')}$ is lower semicontinuous. 
By Proposition 7.31 in \cite{bs},  $(s,a)\to\int_Se^{rv(s')}q(ds'|s.a)$ is also lower semicontinuous. 
Hence,  $(s,a)\to \ln \int_Se^{rv(s')}q(ds'|s.a)$ is lower semicontinuous and  
Since $r<0,$  $(s,a)\to  M(v)(s,a)$ is upper semicontinuous.

Assume that $p\in (0,1).$ Then, the functions $y\to y^p$ and $y\to y^{1/p}$ are increasing
and consequently,
$$ (s,a)\to [(1-\beta)u(s,a)^p+ \beta M(v)(s,a)^p]^{1/p}$$
is upper semicontinuous.

If $p<0,$ then the functions $y\to y^p$ and $y\to y^{1/p}$ are decreasing.
Then, $(s,a)\to  M(v)(s,a)^p$ and $(s,a)\to u(s,a)^p$ 
are lower semicontinuous.  
Since $1/p <0,$ we conclude that 
$$ (s,a)\to [(1-\beta)u(s,a)^p+ \beta M(v)(s,a)^p]^{1/p}$$
is upper semicontinuous. 

The second part follows from Lemma \ref{l1}(b).  $\Box$ \\

\begin{lema} \label{l3c} If $v\in\widehat{\mathcal{C}}$, 
and (C1)-(C3) hold, then the function $(s,a)\to H(s,a,v)$ is continuous.  
Moreover, $\mathcal{T}v\in \widehat{\mathcal{C}}.$
  \end{lema}

\noindent{\bf Proof } The proof proceeds similarly as in Lemma \ref{l3}. The second part is due to
the Maximum Theorem on p. 116 in \cite{berge} or Theorem 17.31 in \cite{ab}. $\Box$

\subsection {Appendix C}

\begin{lema}\label{SH}  The operator $\mathcal T$ is strongly subhomogeneous in the class 
$ \widehat{{ \cal A}}.$  
\end{lema}

\noindent{\bf Proof  } Fix $v\in \widehat{\mathcal{A}}$ and $c\in(0,1).$  
Then, from the proof of Lemma \ref{A2} we infer that the functions $(s,a)\to M(v)(s,a)$ and $(s,a)\to M(cv)(s,a)$ 
are upper semianalytic for every $(s,a)\in\K.$
By the Jensen inequality we have 
\begin{equation}\label{ji}
 M(cv)(s,a) =\frac{1}{r}\ln \int_X e^{crv(y)}q(dy|s,a)
\ge cM(v)(s,a)=\frac{c}{r}\ln \int_S e^{rv(y)}q(dy|s,a)
\end{equation}
for every $(s,a)\in\mathbb{K}.$

\underline{Case: $ p>0.$}   Recall that  $y\to y^p$ and $y\to y^{1/p}$ are increasing and it holds
$\epsilon^p  \le [M(v)(s,a) ]^p \le B^p.$ 
Let 
\begin{eqnarray*}
L(s,a)& :=& (1-\beta)u(s,a)^p + \beta c^p[M(v)(s,a)]^p\quad\quad\quad \mbox{and}\\
R(s,a)&:=& (1-\beta)b^pu(s,a)^p + \beta b^p[M(v)(s,a)]^p
=b^p\left((1-\beta) u(s,a)^p + \beta  [M(v)(s,a)]^p\right).
\end{eqnarray*}
We are going to determine $b\in (c,1)$ such that
$$L(s,a) \ge R(s,a) \quad \mbox{for all} \quad (s,a)\in \mathbb{K}.$$
This inequality is equivalent to  
\begin{equation}
\label{1} (1-\beta)u(s,a)^p(1-b^p) \ge \beta(b^p-c^p)[M(v)(s,a)]^p.
\end{equation}
Note that for all $(s,a)\in \mathbb{K},$
$$
(1-\beta)u(s,a)^p(1-b^p) \ge   (1-\beta)\epsilon^p
(1-b^p) $$
and 
$$ 
\beta(b^p-c^p)B^p  \ge \beta(b^p-c^p)[M(v)(s,a)]^p.
 $$
Therefore (\ref{1}) holds, if
$$
(1-\beta)\epsilon^p(1-b^p) \ge 
\beta(b^p-c^p)B^p 
$$ 
or equivalently
$$
\frac{ (1-\beta)\epsilon^p +\beta c^pB^p  }{
 (1-\beta)\epsilon^p +\beta B^p}\ge b^p. 
$$
Take $\psi(c)$ as the maximum possible value of $b$ satisfying the above inequality, i.e.,
\begin{equation}\label{numer1}
b= \psi(c):= \left(
\frac{ (1-\beta)\epsilon^p +\beta c^pB^p  }{
 (1-\beta)\epsilon^p +\beta B^p} \right)^{1/p}.
\end{equation}
Note that $c<\psi(c)<1$ since $c^p < \psi(c)^p <1$ and $p\in (0,1).$
By (\ref{ji}) we have
$$\mathcal{T}(cv)(s) \ge L(s,a)^{\frac{1}{p}} \ge
R(s,a)^{\frac{1}{p}}
$$
for all $(s,a)\in \mathbb{K}.$ 
Hence ${\cal T}(cv)(s) \ge \sup_{a\in A(s)} R(s,a)^{\frac{1}{p}} =
\psi(c){\cal T}v(s).$  
Thus,  $\mathcal{T}$ is strongly subhomogeneous.

\underline{Case  $p<0.$} Here  $y\to y^p$ and $y\to y^{1/p}$ are decreasing.
It holds  $\epsilon^p  \ge [M(v)(s,a)]^p  \ge B^p$ and 
$${\cal T}(cv)(s) \ge \left((1-\beta)u(s,a)^p +\beta [M(cv)(s,a)]^p\right)^{1/p}$$
for all $(s,a)\in \mathbb{K}.$ Next from (\ref{ji})
$$(1-\beta)u(s,a)^p +\beta [M(cv)(s,a)]^p\le L(s,a):=
(1-\beta)u(s,a)^p +c^p\beta [M(v)(s,a)]^p.$$
We want to find $b\in (c,1)$ such that 
\begin{eqnarray*}
 L(s,a) \le R(s,a)&:=&
(1-\beta)b^pu(s,a)^p + \beta b^p[M(v)(s,a)]^p\\
&= &
 b^p((1-\beta) u(s,a)^p + \beta  [M(v)(s,a))]^p.
 \end{eqnarray*}
This inequality is equivalent to
\begin{equation}
\label{3}
(c^p-b^p) \beta  [M(v)(s,a)]^p\le (1-\beta) u(s,a)^p(b^p -1).
\end{equation}
Since $p<0$ we note that  
$$
(c^p-b^p) \beta  [M(v)(s,a)]^p\le (c^p-b^p) \beta \epsilon^p$$
and 
$$
(1-\beta) B^p (b^p -1)\le (1-\beta) u(s,a)^p(b^p -1).
$$
If we find $b \in (c,1)$ such that
\begin{equation}
\label{4}
(c^p-b^p) \beta \epsilon^p 
  \le (1-\beta) B^p (b^p -1)
\end{equation}
then (\ref{3}) holds.  From (\ref{4}) we get
$$ b^p \ge \frac{(1-\beta) B^p + c^p\beta \epsilon^p }{(1-\beta) B^p + \beta \epsilon^p }.
$$
Put  
\begin{equation}\label{numer2}
b= \psi(c):=
\left(\frac{(1-\beta) B^p + c^p\beta \epsilon^p }{(1-\beta) B^p + \beta \epsilon^p }\right)^{\frac1p}.
\end{equation}
Since $c^p>1$ we have that $\psi(c)^p >1$ and $\psi(c)^p < c^p.$  Hence $\psi (c) \in (c,1).$
Furthermore, by (\ref{ji}) we have
$${\cal T}(cv)(x) \ge L(s,a)^{\frac{1}{p}} \ge R(s,a)^{\frac{1}{p}}
$$
for all $(x,a)\in \mathbb{K}.$
Hence, 
${\cal T}(cv)(s) \ge \sup_{a\in A(s)} R(s,a)^{\frac{1}{p}}=
\psi(c){\cal T}v(s),$
that is,  $\mathcal{T}$ is strongly subhomogeneous.
 $\Box$\\

\noindent{\bf Proof of Theorem \ref{thm2} }\ 
By Theorem \ref{thm1}, the Bellman equation has  a solution  $v^* \in \widehat{{\cal A}}.$ 
 Take any $v_0 \in \widehat{{\cal A}}. $ 
Notice that for $t_0:= \epsilon/B$   we have
\begin{equation}
\label{preq1}
t_0 v^*(s)\le \epsilon \le  \ v_0(s) \le  B \le v^*(s)/t_{0}
\end{equation}
for all $s\in S.$  By Lemma \ref{SH}, $\cal T$  is strongly subhomogeneous, which means that 
there exists a continuous increasing function
$\psi:(0,1)\to (0,1)$ such that  $\psi( t_0)\in (t_0,1)$ and by (\ref{preq1})
$$ \psi(t_0)v^* = 
\psi(t_0){\cal T}v^*\le {\cal T}(t_0v^*)  
\le {\cal T}v_0.$$
We also have $t_0v_0\le v^*$ and therefore
$$
\psi(t_0){\cal T}v_0 \le {\cal T}(t_0v_0) 
\le {\cal T}v^*= v^*.$$
Hence, we have  ${\cal T}v_0 \le v^*/\psi(t_0)$ and
consequently
  $$\psi(t_0)v^*\le {\cal T}v_0\le 
\frac{v^*}{\psi(t_0)}.$$
Setting $t_1:=\psi(t_0)$ we get 
$$t_1v^*\le {\cal T}v_0\le v^*/t_1. $$
Put $t_{k+1}:=\psi( t_k),$ $k\in\N.$ Using this fact and  
iterating the above inequality and observing that the sequence $(t_k)$ is increasing  
we obtain
\begin{equation}\label{fin}
t_kv^*\le {\cal T}^{(k)}v_0\le v^*/t_k.
\end{equation}
As $(t_k)$ is bounded from above by one, $t^*:= \lim_{k\to\infty} t_k$ exists.
Note that then by the continuity and  monotonicity of $\psi $ we get
$$\lim_{k\to\infty} t_k =t^*=\lim_{k\to\infty} \psi(t_{k-1})=\psi(t^*).$$
Hence, $t^*=1.$ Otherwise,  from the strong homogeneity property  
we could choose again  $\psi( t^*)\in (t^*,1)$. From (\ref{fin}), we conclude
$$\left(1-\frac{1}{t_k}\right)v^*(s)\le(t_{k}-1)v^*(s)\le {\cal T}^{(k)}v_0(s)-v^*(s)\le 
\left(\frac{1}{t_k}-1\right)v^*(s)$$
for all $s\in S.$ Therefore, 
$$ \|\mathcal{T}^{(k)}v_0-v^*\|\le B \left(\frac{1}{t_{k}}-1\right),$$
 and consequently 
\begin{equation}
\label{2sol}\| {\cal T}^{(k)}v_0- v^*\|\to 0 
\ \ \mbox{  as}\ \  k\to \infty.
\end{equation}
Suppose now that $v_1^*\in\widehat{\mathcal{A}}$ is a solution to the Bellman equation 
different from $v^*.$ Substituting $v_0=v_1^*$
in (\ref{2sol}) and using equality ${\cal T}^{(k)}v_1^*=
v_1^*$ we get that $\| v_1^*-v^*\|=0.$ Hence, we must have $v_1^=v^*.$ $\Box$

\subsection{Appendix D}

\noindent{\bf Proof of Theorem \ref{thm4} } We know that $U(f)$ satisfies equation   (\ref{fpp1}). The operator
$\mathcal{T}_f$ is strongly subhomogeneous. The proof proceeds along the similar lines as the proof of Lemma \ref{SH}
with $\mathcal{T}$ and $\widehat{{\mathcal A}}$ replaced by $\mathcal{T}_f$ and  $\widehat{{\mathcal M}},$ 
respectively. The uniqueness mimics the proof of  
Theorem \ref{thm2} with the operator $\mathcal{T}_f$. $\Box$\\

\noindent{\bf Proof of Theorem \ref{thm3} }
Assume (U1)-(A3).
By Theorem \ref{thm1u} there is a unique 
 $v^* \in \widehat{\mathcal{U}}$ and by Lemma \ref{l1}(b) there exists 
$f^*\in \Phi$ such that
$$ v^*(s)= {\cal T}v^*(s) = H(s,f^*(s),v^*)\quad\mbox{for all}\quad s\in S.$$
By Theorem \ref{thm4}, we get 
$v^*(s)=U(f^*)(s)$ for all $s\in S.$ 
From the proof of Theorem \ref{thm1u}, we know that
$$v^*(s)= \lim_{k\to\infty}{\cal T}^{(k)}B(s),\quad\  s\in S.$$
Let $\pi=(f_n)$ be any Markov policy. 

\underline{ Case $ p>0.$}  
For each 
$s\in S$ and  $k\in \mathbb{N},$ we have
$$
{\cal T}^{(k)}B(s)\ge 
{\cal T}_{f_1}\cdots {\cal T}_{f_k}B(s)
\ge {\cal T}_{f_1}\cdots {\cal T}_{f_k}{\bf 0}(s)=
 U_k(\pi)(s),$$
where the second inequality is only valid for positive $p$.
Hence, by Lemma \ref{nierutil}(b) we let $k\to \infty$ and  
$$v^*(s)\ge \lim_{k\to\infty} 
U_k(\pi)(s),\quad  \ s\in S.$$
Thus, we have shown that
$$v^*(s)=U(f^*)(s)\ge U(\pi)(s)$$
for each policy $\pi\in\Pi_M$ and for all $s\in S.$ From Remark \ref{all} it follows that
$v^*\ge U(\pi)$ for any policy $\pi\in\Pi.$

\underline{Case $ p<0.$}   
For each $n\in \mathbb{N},$ put 
$$w_n^*(s):= \sup_{\pi\in\Pi} U_n(\pi)(s),\quad \ s\in S.$$
By standard dynamic programming techniques  \cite{hl} and Lemma \ref{l3}, we conclude
that $w_n^*$ is upper semicontinuous. From Lemma \ref{nierutil}(a), 
it follows that the sequence $(w_n^*)$ is monotone decreasing. By Remark \ref{stala}, $\epsilon\le w_n^*\le B$
for every $n\in\N.$
Thus, the limit $w^*(s)=\lim_{n\to\infty}w_n^*(s),$ for $s\in S,$ 
exists and by Lemma 2.41 in \cite{ab} the function $w^*$ is upper semicontinuous. Moreover, we have
$$w_{n+1}^*(s) ={\cal T}w^*_n(s),\ \ s\in S,\ n\in \mathbb{N}.$$ 
From this  equation and Lemma \ref{l2}, we infer that
$w^*={\cal T}w^*$. Next, by Lemma \ref{l1}(b) there exists $f^*\in \Phi$ such that
$$w^*(s)= {\cal T}w^*(s) ={\cal T}_{f^*}w^*(s)\ \ \mbox{for all}\ \ s\in S.$$ 
By Theorems \ref{thm1u} and \ref{thm4} we have that $w^*(s) = v^*(s)= U(f^*)(s)$ for all $s\in S.$ 
Now observe that, for each $n\in\mathbb{N}$ and $\pi\in \Pi_M$,
$$w_n^*(s) \ge {\cal T}_{f_1}\cdots {\cal T}_{f_n}{\bf 0}(s)=U_n(\pi)(s),\ \ s\in S.$$ Hence,
$$U(f^*)(s)=\lim_{n\to \infty} w_n^*(s) \ge
\lim_{n\to \infty} U_n(\pi)(s)=U(\pi)(s),\ \  s\in S.$$
Again from  Remark \ref{all} we conclude  that
$v^*\ge U(\pi)$ for any policy $\pi\in\Pi.$

Thus, in both cases we have shown that $f^*$ is an optimal stationary policy. 

For (C1)-(C3) the proof proceeds analogously. 
$\Box$

\end{document}